\def\bu{\bullet}
\def\marker{\>\hbox{${\vcenter{\vbox{
    \hrule height 0.4pt\hbox{\vrule width 0.4pt height 6pt
    \kern6pt\vrule width 0.4pt}\hrule height 0.4pt}}}$}\>}
\def\gpic#1{#1
     \smallskip\par\noindent{\centerline{\box\graph}} \medskip}
\documentclass[12pt]{article}
\usepackage{fullpage,amsthm,amssymb,amsmath}
\usepackage{graphicx,algorithmic,algorithm,amsfonts}
\usepackage{enumerate,color}
\usepackage{tikz}

\newtheorem{thm}{Theorem}[section]
\newtheorem{lem}[thm]{Lemma}
\newtheorem{prop}[thm]{Proposition}
\newtheorem{cor}[thm]{Corollary}

\theoremstyle{definition}
\newtheorem{example}[thm]{Example}
\newtheorem{definition}[thm]{Definition}
\newcommand\GG{\mathcal{G}}
\newcommand\NN{\mathbb{N}}

\def\qed{\hfill\ifhmode\unskip\nobreak\fi\quad\ifmmode\Box\else\hfill$\Box$\fi}

\def\FR#1#2{\frac{#1}{#2}}
\def\CL#1{\lceil{#1}\rceil}

\def\esub{\subseteq}

\def\VEC#1#2#3{#1_{#2},\ldots,#1_{#3}}

\begin{document}

\title{Largest 2-regular subgraphs in 3-regular graphs}

\author{
Ilkyoo Choi\thanks{
Department of Mathematics, Hankuk University of Foreign Studies, Yongin-si,
Gyeonggi-do, Republic of Korea, \texttt{ilkyoo@hufs.ac.kr}.
Supported by the Basic Science Research Program through
the National Research Foundation of Korea (NRF) funded by the Ministry of
Education (NRF-2018R1D1A1B07043049), and also by Hankuk University of Foreign
Studies Research Fund.
}
\and
Ringi Kim\thanks{
Department of Mathematical Sciences, KAIST, Daejeon, Republic of Korea,
\texttt{ringikim2@gmail.com}.
Supported by the National Research Foundation of Korea (NRF)
grant funded by the Korea government (MSIT)(NRF-2018R1C1B6003786)
}
\and
Alexandr Kostochka\thanks{
Department of Mathematics, University of Illinois at Urbana--Champaign, Urbana,
IL, USA and Sobolev Institute of Mathematics, Novosibirsk, Russia,
\texttt{kostochk@math.uiuc.edu}.
Supported by NSF grants DMS1600592 and grants 18-01-00353A and
16-01-00499 of the Russian Foundation for Basic Research.
}
\and
Boram Park\thanks{
Department of Mathematics, Ajou University, Suwon-si, Gyeonggi-do,
Republic of Korea, \texttt{borampark@ajou.ac.kr}.
Supported by the Basic Science Research Program through the National Research
Foundation of Korea (NRF) funded by the Ministry of Science, ICT and
Future Planning (NRF-2018R1C1B6003577).
}
\and
Douglas B. West\thanks{
Mathematics Departments, Zhejiang Normal University, Jinhua, China and
University of Illinois at Urbana--Champaign, Urbana, IL, USA,
\texttt{dwest@math.uiuc.edu}.
Supported by National Natural Science Foundation of China grant NSFC-11871439
and Recruitment Program of Foreign Experts, 1000 Talent Plan, State
Administration of Foreign Experts Affairs, China.
}
}

\date\today

\maketitle

\begin{abstract}
For a graph $G$, let $f_2(G)$ denote the largest number of vertices in a
$2$-regular subgraph of $G$.  We determine the minimum of $f_2(G)$ over
$3$-regular $n$-vertex simple graphs $G$.  To do this, we prove that every
$3$-regular multigraph with exactly $c$ cut-edges has a $2$-regular subgraph
that omits at most $\max\{0,\lfloor (c-1)/2\rfloor\}$ vertices.
More generally, every $n$-vertex multigraph with maximum degree $3$ and
$m$ edges has a $2$-regular subgraph that omits at most
$\max\{0,\lfloor (3n-2m+c-1)/2\rfloor\}$ vertices.  These bounds
are sharp; we describe the extremal multigraphs.

\end{abstract}

\section{Introduction}
For $\ell\in\NN$, an {\it $\ell$-factor} in a graph or multigraph is an
$\ell$-regular spanning subgraph.  Let $f_i(G)$ denote the maximum number of
vertices in an $i$-regular subgraph of $G$.  A graph or multigraph is 
{\it cubic} if every vertex has degree $3$.

A classical theorem by Petersen~\cite{1891Pe} says that every cubic multigraph
with at most two cut-edges has a $2$-factor and (equivalently) a $1$-factor.
Thus $f_1(G)=f_2(G)=|V(G)|$ when $G$ is $3$-regular and has at most two cut-edges.  In this paper, we extend
this result on $f_2(G)$ to the setting where there are more cut-edges and also
to the setting of maximum degree $3$.

For a $(2r+1)$-regular graph $G$ with $n$ vertices, Henning and Yeo~\cite{HY}
proved $f_1(G)\ge n-r\FR{(2r-1)n+2}{(2r+1)(2r^2+2r-1)}$ (while studying
matchings), and this is sharp.
The formula reduces to $(8n-2)/9$ for $3$-regular graphs.
O and West~\cite{2010OWe} gave a short proof of the Henning--Yeo result
using the notion of a {\it balloon} in a graph, which they defined to be a
maximal $2$-edge-connected subgraph incident to exactly one cut-edge.

We use balloons to study the minimum of $f_2(G)$ when $G$ is $3$-regular with
$n$ vertices.  For $3$-regular graphs, the notion of balloon has a simpler
equivalent description: a graph obtained from a $2$-edge-connected $3$-regular
graph by subdividing one edge.

In order to solve the problem, we consider a more general question, determining
a sharp lower bound on $f_2(G)$ in terms of the number of cut-edges in $G$.
Our basic result is

\begin{thm}\label{thm:cubic}
If $G$ is a cubic $n$-vertex multigraph with $c$ cut-edges,  then
$f_2(G)\ge n-\max\{0,\left\lfloor{c-1\over 2}\right\rfloor\}$, and this bound
is sharp. 
\end{thm}

We will also describe all the multigraphs that achieve equality in the
bound.  Since O and West~\cite{2010OWe} showed that a cubic $n$-vertex graph
has at most $(n-7)/3$ cut-edges, Theorem~\ref{thm:cubic} immediately yields a
lower bound on $f_2(G)$ for a cubic graph $G$ in terms of the number of
vertices alone.  It also yields a somewhat weaker guarantee for cubic
loopless multigraphs.

\begin{cor}\label{cor:cubic}
If $G$ is a cubic $n$-vertex graph, then $f_2(G)\ge\min\{n,\CL{\FR56(n+2)}\}$.
If $G$ is a cubic $n$-vertex loopless multigraph, then
$f_2(G)\ge \min\{n,\CL{\FR34(n+2)}\}$.  Both bounds are sharp.
\looseness -1
\end{cor}

Theorem~\ref{thm:cubic} is proved more simply by considering the broader class
of {\em subcubic} multigraphs, which are those having maximum degree at most
$3$.  Given an $n$-vertex multigraph $G$ with maximum degree at most $2r+1$,
the {\it $r$-deficit} of $G$ is the difference between $(2r+1)n$ and the
degree-sum of $G$, which can be computed as $(2r+1)n-2|E(G)|$.

\begin{thm}\label{thm:subcubic}
If $G$ is a subcubic $n$-vertex multigraph with $c$ cut-edges and $1$-deficit
$d$, then $f_2(G)\ge n-\max\{0,{d+c-1\over 2}\}$, and this bound is sharp.
\end{thm}

Several constructions of sharpness examples together lead to a characterization
of all sharpness examples.

\begin{example} \label{ex:sharp}
{\it Trees.}
A subcubic $n$-vertex tree has $n-1$ cut-edges.  Its $1$-deficit is
$3n-2(n-1)$, so in this case $(d+c-1)/2=n$.  Hence Theorem~\ref{thm:subcubic}
guarantees nothing, and in fact a tree has no $2$-regular subgraph.

{\it Balloons.}
By definition, a balloon has no cut-edge and has $1$-deficit $1$.
Theorem~\ref{thm:subcubic} guarantees a $2$-factor, which achieves equality
in the bound.

{\it Bipartite multigraphs.}
Let $H$ be a $2$-connected cubic bipartite multigraph with parts $X$ and
$Y$; note that $|X|=|Y|$.  Let $G=H-\hat y$, where $\hat y\in Y$.  If $G$ is
$2$-connected, then $G$ has no cut-edge and has $1$-deficit $3$.
Since the number of vertices in $G$ is odd and all cycles in $G$ are even,
$G$ has no $2$-factor.  Theorem~\ref{thm:subcubic} guarantees a $2$-regular
subgraph in $G$ with $n-1$ vertices, where $n=|V(G)|$.  Hence $G$ is a
sharpness example.
\end{example}

The argument for bipartite multigraphs in Example~\ref{ex:sharp} applies to
confirm sharpness for a larger family.  (Recall that a cubic multigraph is
$2$-connected if and only if it has no cut-edge, with the exception of the
loopless multigraph with two vertices and three edges.)

\begin{definition}\label{GGdef}
Let $\GG$ be the family of multigraphs obtained in the following way: 

(1) Start with a $2$-connected cubic bipartite multigraph $H$ with parts $X$ and
$Y$.

(2) Delete one vertex $\hat y\in Y$ such that $H-\hat y$ is $2$-connected. 

(3) Explode (or not) each vertex $y$ in $Y-\hat y$, where {\it exploding} $y$
means taking the disjoint union of the current graph with a $2$-connected cubic
multigraph $F$ and then replacing both $y$ and a vertex $z$ in $F$ with three
edges joining the neighborhoods of $y$ and $z$ so that all vertices have
degree $3$.
\end{definition}

We will show that combining sharpness examples via cut-edges preserves
sharpness.  Trees are assembled in this way from single vertices, so we do
not need them as fundamental building blocks for sharpness examples.  With
the characterization of sharpness, our main result (including all those
mentioned previously and proved in Section 2) is then the following.

\begin{thm}\label{thm:subcubic'}
If $G$ is a subcubic $n$-vertex multigraph with $c$ cut-edges and
$1$-deficit $d$, then $f_2(G)\ge n-\max\{0,{d+c-1\over 2}\}$.
When $G$ is connected, equality holds if and only if each component after
deleting all the cut-edges is a single vertex, a balloon, or a graph in $\GG$.
\end{thm}

In Section 3, we offer additional enhancements.  First, we generalize
by restricting to graphs with girth at least $g$.  Second, we show that one
can restrict the initial bipartite multigraph $H$ in the definition of $\GG$
by forbidding multi-edges.  Third, one can alternatively restrict each
multigraph $F$ used to explode a vertex to be factor-critical, where
{\it factor-critical} means having a matching that omits only any one vertex.
However, one cannot ensure these latter two enhancements simultaneously.

Generalizing the problem, one would seek first a large $2$-regular subgraph
when $G$ is $(2r+1)$-regular, and then more generally a large $2k$-regular
subgraph when $G$ is $(2r+1)$-regular.  It is reasonable to think that
$f_2(G)\ge n-\max\{0,\CL{d+c-1\over 2r}\}$ holds when $G$ has maximum degree
$2r+1$ with $c$ cut-edges and $r$-deficit $d$, because sharpness holds in
two quite different classes.  Equality holds for trees ($n-1$ cut-edges,
$r$-deficit $(2r-1)n+2$, and no $2$-regular subgraph) and for $(2r+1)$-regular
graphs with at most $2r$ cut-edges (Hanson, Loten, and Toft~\cite{HLT} showed
that every such graph has a $2$-factor).

For the general problem of minimizing $f_{2k}(G)$ when $G$ is $(2r+1)$-regular,
Kostochka et al.~\cite{arxiv_KRTWZ} generalized~\cite{HLT} by showing that if
$k<(2r+1)/3$ and $G$ has at most $2r-3(k-1)$ cut-edges, then $G$ has a
$2k$-factor.  Therefore, we are interested in how large a $2k$-regular subgraph
is guaranteed when there are more cut-edges.  In this paper, we settle the case
$k=r=1$.

\section{The Main Result}

To prove the desired bound on the number of vertices omitted by a largest
$2$-regular subgraph, in cases where the graph has no cut-edge we will need two
earlier results.

First, a result of Edmonds~\cite{1965Ed} easily implies the following lemma.

\begin{lem}[O and West~\cite{2015OWe}]\label{lem:weight}
Every edge-weighted 2-edge-connected 3-regular multigraph has a perfect
matching containing at most 1/3 of the total weight.
\end{lem}

The results of Edmonds~\cite{1965Ed} were used earlier in an essentially
equivalent way by Naddef and Pulleyblank~\cite{NP} to prove that every
edge-weighted $(t-1)$-edge-connected $t$-regular multigraph of even order
has a $1$-factor with weight at least a fraction $1/t$ of the total weight.
Here the {\it order} of a graph is its number of vertices.

We also use a special case for cubic graphs of a result of Plesn\'{\i}k that
strengthens the usual conclusion about $1$-factors in regular graphs.

\begin{lem}[Plesn\'{\i}k~\cite{Ples}]\label{plesnik}
Every $(t-1)$-edge-connected $t$-regular graph of even order has a $1$-factor
that avoids any $t-1$ specified edges.
\end{lem}

We also use the special case of Tutte's $1$-Factor Theorem~\cite{Tu} for
$3$-regular multigraphs $G$, stating that if $G$ has no $1$-factor, then $V(G)$
contains a {\it Tutte set} $S$ such that $o(G-S)\ge|S|+2$, where $o(H)$ is the
number of components of $H$ having odd order, called {\it odd components}.

We can now prove the main result, which we restate for ease of reference.


\begin{thm}\label{thm:subcubic''}
If $G$ is a subcubic $n$-vertex multigraph with $c$ cut-edges and
$1$-deficit $d$, then $f_2(G)\ge n-\max\{0,{d+c-1\over 2}\}$.
When $G$ is connected, equality holds if and only if each component after
deleting all the cut-edges is a single vertex, a balloon, or a graph in $\GG$.
\end{thm}

\begin{proof}
By Petersen's Theorem~\cite{1891Pe}, a cubic graph $G$ with at most two
cut-edges has a 2-factor.  Hence we may assume $d>0$ or $c>2$.  Indeed,
we may assume this in each component.  Hence we may also assume that $G$
is connected and must prove $f_2(G)\ge n-(d+c-1)/2$.

The difficult case is when $c=0$ and $d>0$.  We postpone this basis step for a
proof by induction on the number of cut-edges, considering first the induction
step.

Deleting a cut-edge $e$ from $G$ leaves its endpoints with degree less than $3$.
Letting $G-e$ be the disjoint union of $G_1$ and $G_2$, each containing an
endpoint of $e$.  For $i\in\{1, 2\}$, let $c_i$ be the number of cut-edges and
$d_i$ be the $1$-deficit of $G_i$.  Since neither $G_1$ nor $G_2$ is
$3$-regular, and both are subcubic with fewer cut-edges than $G$, the
induction hypothesis applies to each.

That is, $G_i$ has a 2-regular subgraph $H_i$ omitting at most
$(d_i+c_i-1)/2$ vertices, and the disjoint union $H_1+H_2$ is a 2-regular
subgraph of $G$ omitting at most $(d_1+d_2+c_1+c_2-2)/2$ vertices.  Since
$d = d_1+d_2-2$ and $c=c_1+c_2+1$, the graph $H_1+H_2$ omits at most
$(d+c-1)/2$ vertices of $G$.  Equality holds if and only if it holds in both
$G_1$ and $G_2$, which implies inductively that equality holds in $G$ if and
only if $G$ has the claimed description.

\smallskip
Now consider the basis step: $G$ has no cut-edge, but $d>0$.  If $G$ has 
only one vertex, then the formula holds with equality whether the vertex has
a loop or not.  Hence we are reduced to a connected subcubic multigraph with
more than one vertex.

Since $G$ has no cut-edge, $G$ now has minimum degree $2$.  We may also assume
that $G$ has maximum degree $3$, since $f_2(G)=n$ when $G$ is $2$-regular.  We
use Lemma~\ref{lem:weight}.  A {\it thread} in a graph is a maximal path whose
internal vertices have degree $2$ (it may have just one edge); the endpoints of
each thread in $G$ have degree $3$.  Let a {\it $j$-vertex} be a vertex of
degree $j$.

Suppress each $2$-vertex of $G$ by turning each thread through $2$-vertices
into one weighted edge whose weight equals the length of the thread.  The total
weight of the resulting graph $G'$ is the number of edges in $G$.  Deleting
from $G'$ the matching guaranteed by Lemma~\ref{lem:weight} leaves a $2$-factor
of $G'$ whose total weight is at least $2/3$ of the total weight of $G'$.  This
$2$-factor expands back into a $2$-regular subgraph of $G$ that has at least
$2/3$ of the edges of $G$.

\smallskip
Hence $G$ has a $2$-regular subgraph $H$ with at least $2m/3$ vertices, where
$m = |E(G)|$.  Let $t = |V(H)|$.  Since $d = 3n-2m$, we have
$n-t \le n-(2m/3) = d/3$.  If $d>3$, then $d/3 < (d-1)/2 = (d+c-1)/2$, so here
the bound holds and cannot hold with equality.

If $d\in\{1,2\}$, then the formula requires a $2$-factor.  Suppressing the
$2$-vertex or the two $2$-vertices leaves a $3$-regular graph $G'$ with no
cut-edge.  By Lemma~\ref{plesnik}, the graph $G'$ has a $1$-factor that omits
the edge(s) formed by suppressing $2$-vertices.  Deleting this $1$-factor
leaves a $2$-factor in $G'$ that uses those edge(s), and it expands to a
$2$-factor in $G$.  When $d=2$, equality cannot hold in the formula, since the
formula is not an integer.  When $d=1$, equality holds, and $G$ is a balloon,
as claimed.

\smallskip
Finally, assume $d=3$.  At each of the three $2$-vertices of $G$, add a
cut-edge and a balloon to form a $3$-regular graph $G'$.  If $G'$ has a
$1$-factor, then deleting its edges (and the added vertices) leaves a
$2$-factor of $G$.  Otherwise, $G'$ has a Tutte set $S$ such that
$o(G'-S)\ge|S|+2$.  By parity of the degree-sum, an odd number of edges join
$S$ to any odd component of $G'-S$.

Let $m$ be the number of edges joining $S$ to $V(G'-S)$; note that $m\le 3|S|$.
Since $G$ has no cut-edge, each odd component of $G'-S$ other than an added
balloon receives at least three edges from $S$.  Therefore,
$m\ge 3+3(|S|-1)$, and equality must hold.  Since $G'$ is connected, also
$G'-S$ has no even components, the components of $G'-S$ are the added balloons
and others receiving exactly three edges, and $S$ is an independent set.

The components of $G'-S$ other than the added balloons are the set $T$ of
components of $G-S$, each having odd order.
The edges in $G$ joining $S$ to $T$ form a bipartite multigraph $F$ with parts
$S$ and $T$ obtainable by deleting one vertex of a 3-regular bipartite graph
(which produces the three $2$-vertices in $G$).  To obtain $G$ from $F$, each
vertex of $T$ is left alone or is exploded.  Thus every extremal graph with
$d=3$ has the form described.

\smallskip
Also every such graph is extremal.  To prove this, it remains only to show
that every graph $G\in\GG$ has no $2$-factor.  The construction of $G$
according to Definition~\ref{GGdef} begins with a bipartite graph $H$ having
parts $X$ and $Y$.  A vertex $y\in Y-\{\hat y\}$ may be exploded using a
$2$-connected multigraph $F$, but the vertices of $F-z$ that are made adjacent
to the neighborhood of $y$ in $H$ lie in the same component of $G-X$.

Suppose that $G$ has a $2$-factor and orient each cycle consistently.
Each vertex of $X$ is followed on its cycle by a vertex that corresponds to a
particular vertex $y$ in $Y$, and the cycle can only leave that component of 
$G-X$ via a vertex corresponding to the same vertex $y$.  Also, since $H$
is $3$-regular, that vertex $y$ cannot serve in this way for any other vertex
$x\in X$.  Since $|X|>|Y-\{\hat y\}|$, there cannot be disjoint
cycles covering all the vertices of $X$.
\end{proof}

\section{Enhancements}

In this section, we consider several refinements of the main result.

As noted in Corollary~\ref{cor:cubic}, Theorem~\ref{thm:subcubic'} specializes
for cubic graphs ($d=0$) to say for $n>4$ that every cubic $n$-vertex graph has
a $2$-regular subgraph with at least $5(n+2)/6$ vertices; this uses that
such a graph has at most $(n-7)/3$ cut-edges~\cite{2010OWe}.

\begin{example}\label{cut-edge}
Equality holds in Corollary~\ref{cor:cubic} for every graph $G$ obtained by
starting with a tree whose internal vertices all have degree $3$ and attaching
a $5$-vertex balloon at each leaf.  When all internal vertices have degree $3$,
the number of leaves in the tree exceeds the number of internal vertices by $2$.
The internal vertices lie in no cycle and hence in no $2$-regular subgraph,
while the balloons have $2$-factors.  With $t$ internal vertices and $n$
vertices altogether, we have $n=t+5(t+2)$ and $f_2(G)=5(t+2)$, so
$f_2(G)=5(n+2)/6$.

Furthermore, equality holds in $f_2(G)\ge3(n+2)/4$ for cubic multigraphs by
using the balloon obtained by subdividing one edge of a triple-edge instead
of the $5$-vertex simple balloon.  In both cases, this describes
all examples achieving equality (see~\cite{2010OWe}).
\end{example}

We can generalize Corollary~\ref{cor:cubic} and Example~\ref{cut-edge}
in terms of girth by considering the minimum number of vertices in a balloon
with girth $g$.  When $g\ge2$, a smallest balloon with girth $g$ arises from a
smallest $3$-regular (multi)graph with girth $g$ by subdividing one edge.  We
can pick an edge to subdivide that does not increase the girth as long as there
is an edge that does not belong to every shortest cycle.  Such an edge exists
because the vertex degrees are not $2$.

A smallest $k$-regular graph with girth $g$ is called a {\it $(k,g)$-cage} (for
$g=2$ it consists of two vertices joined by $k$ edges).  Determining the
minimum number $h(k,g)$ of vertices in a $(k,g)$-cage is a well-known and very
difficult problem.  The smallest balloon with girth $g$ will have $h(3,g)+1$
vertices.  For $g\in\{2,\ldots,12\}$, the number of vertices is
3, 5, 7, 11, 15, 25, 31, 59, 71, 113, 127, respectively (see~\cite{EJ},
for example).

\begin{cor}
If $G$ is a cubic $n$-vertex multigraph with girth $g$, then
$$f_2(G)\ge\min\left\{n,{\FR{h'}{h'+1}(n+2)}\right\},$$
where $h'=h(3,g)+1$.  The bound is sharp.  All examples achieving equality
arise by attaching a smallest balloon with girth $g$ at each leaf of a tree
whose internal vertices have degree $3$.
\end{cor}
\begin{proof}
(Sketch) The argument for the upper bound of O and West~\cite{2010OWe} on the
number of cut-edges depends on the smallest order of balloons.  The number of
cut-edges is maximized by attaching smallest subcubic balloons of girth $g$ 
to the leaves of a tree with internal vertices of degree $3$, which yields
the given formula.

The proof is inductive.  Achieving equality for a larger graph requires
achieving equality in both graphs obtained by deleting a cut-edge.  This leads
to the structure described.
\end{proof}

Next we refine Theorem~\ref{thm:subcubic'} by showing that $\GG$ can be
produced in a more restricted way.

\begin{prop}\label{GGsimple}
In Definition~\ref{GGdef} for the family $\GG$, the initial bipartite multigraph
$H$ generating any member of $\GG$ can be taken to be simple, without changing
the resulting family.
\end{prop}
\begin{proof}
For every multigraph in $\GG$, the $1$-deficit is $3$, there is no cut-edge,
and every largest $2$-regular subgraph omits exactly one vertex, as proved in
Theorem~\ref{thm:subcubic''}.

Hence in Theorem~\ref{thm:subcubic''} the graphs in $\GG$ arise only in the
case $d=3$ when the augmented graph $G'$ has no $1$-factor.  As described
there, for any Tutte set $S$ in $G'$, the edges joining $S$ and $T$ form a 
bipartite multigraph that can serve as the multigraph $H$ in the construction
of $G$ as a graph in $\GG$.

To ensure that $H$ can be chosen to be simple, we let $S$ be a minimal Tutte
set in $G'$.  It is an elementary exercise that for any minimal Tutte set $S$
in a cubic multigraph $G'$, each vertex in $S$ has all its neighbors in
distinct components of $G'-S$.  (If the neighbors of any $x\in S$ are confined
to fewer than three odd components of $G'-S$, then deleting $x$ from $S$
reduces $|S|$ by as much as it reduces the number of resulting odd components,
thereby yielding a smaller Tutte set.)

Since the neighbors of each $x\in S$ are in distinct components of $G'-S$,
the neighbors of $x$ in the resulting bipartite multigraph $H$ are distinct
vertices of $Y-\hat y$.
\end{proof}

Finally, the famous Gallai--Edmonds Structure Theorem~\cite{GEE,GEG}
that describes all largest matchings in a multigraph leads to another
refinement of the structure of members of $\GG$.

\begin{definition}
In a multigraph $G$, let $B$ be the set of vertices that are covered by every
maximum matching in $G$.  Let $A$ be the set of vertices in $B$ having at least
one neighbor outside $B$, let $C=B-A$, and let $D=V(G)-B$.  The
{\em Gallai--Edmonds Decomposition} of $G$ is the partition of $V(G)$ into
the three sets $A,C,D$.  The {\em deficiency} ${\rm def}(G)$ of a graph $G$ is
$\max_{S\esub V(G)}\{o(G-S)-|S|\}$.  
\end{definition}

\begin{thm}[Gallai--Edmonds Structure Theorem]
Let $A,C,D$ be the Gallai--Edmonds Decomposition of a multigraph $G$.
Let $\VEC G1q$ be the components of $G-A-C$.  If $M$ is a maximum matching
in $G$, then the following properties hold.

a) $M$ covers $C$ and matches $A$ into distinct components of $G-A-C$.

b) Each $G_i$ is factor-critical.

c) $o(G-A)-|A|={\rm def}(G)=q-|A|$.
\end{thm}

\begin{prop}\label{GGfcrit}
In the construction of any graph $G\in\GG$, the bipartite multigraph $H$
with parts $X$ and $Y$ can be chosen so that each component of $G-X$ is
a factor-critical graph.
\end{prop}
\begin{proof}
In the Gallai--Edmonds Decomposition $(A,C,D)$ of a graph not having a 
$1$-factor, the set $A$ is a Tutte set.  For the auxiliary 3-regular multigraph
$G'$ in the proof of Theorem~\ref{thm:subcubic''} when the $1$-deficit is $3$,
take the Tutte set $S$ to be $A$ in the Gallai--Edmonds Decomposition.
As argued in the proof of Theorem~\ref{thm:subcubic''}, we have $C$ empty and
$A$ independent.  By the Gallai--Edmonds Structure Theorem, with this choice of
the Tutte set and the resulting bipartite graph $H$, the set $A$ becomes $X$,
and the components of $G-X$ are factor-critical.
\end{proof}

\begin{example}\label{badgraph}
The refinements in Propositions~\ref{GGsimple} and~\ref{GGfcrit} cannot
be guaranteed simultaneously (that is, using one initial bipartite graph $H$).
An example showing this appears in Figure 1, shown in solid edges.  This is a
bipartite multigraph $G$ obtained from the complete bipartite graph $K_{2,3}$
by replacing one edge with a thread of length $3$ and then duplicating the
middle edge $xy$ of that thread to reach degree $3$ at its endpoints. 

The augmented graph $G'$ (including the dashed edges) grows a cut-edge from
each $2$-vertex and adds a balloon at the other end of each cut-edge.  Every
maximum matching in $G'$ covers all the vertices of $X$.  Using $H-\hat y$ as
the full multigraph $G$, with no vertices exploded, the components of $G-X$ are
factor-critical, but this $H$ is not simple.

The Tutte set $X$ has size $4$.  Also $X-\{x\}$ is a Tutte set.  This 
Tutte set constructs $G$ by starting with $H-\hat y=K_{2,3}$ and exploding 
one vertex of $Y$ by using the $2$-connected $3$-regular multigraph consisting
of a $4$-cycle with two opposite edges duplicated.
\end{example}
\begin{figure}[h!]
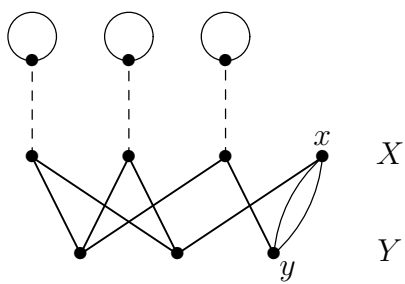

\gpic{
\expandafter\ifx\csname graph\endcsname\relax \csname newbox\endcsname\graph\fi
\expandafter\ifx\csname graphtemp\endcsname\relax \csname newdimen\endcsname\graphtemp\fi
\setbox\graph=\vtop{\vskip 0pt\hbox{%
    \graphtemp=.5ex\advance\graphtemp by 1.266in
    \rlap{\kern 0.380in\lower\graphtemp\hbox to 0pt{\hss $\bu$\hss}}%
    \graphtemp=.5ex\advance\graphtemp by 1.266in
    \rlap{\kern 0.886in\lower\graphtemp\hbox to 0pt{\hss $\bu$\hss}}%
    \graphtemp=.5ex\advance\graphtemp by 1.266in
    \rlap{\kern 1.392in\lower\graphtemp\hbox to 0pt{\hss $\bu$\hss}}%
    \graphtemp=.5ex\advance\graphtemp by 0.759in
    \rlap{\kern 0.127in\lower\graphtemp\hbox to 0pt{\hss $\bu$\hss}}%
    \graphtemp=.5ex\advance\graphtemp by 0.759in
    \rlap{\kern 0.633in\lower\graphtemp\hbox to 0pt{\hss $\bu$\hss}}%
    \graphtemp=.5ex\advance\graphtemp by 0.759in
    \rlap{\kern 1.139in\lower\graphtemp\hbox to 0pt{\hss $\bu$\hss}}%
    \graphtemp=.5ex\advance\graphtemp by 0.759in
    \rlap{\kern 1.646in\lower\graphtemp\hbox to 0pt{\hss $\bu$\hss}}%
    \special{pn 11}%
    \special{pa 380 1266}%
    \special{pa 127 759}%
    \special{fp}%
    \special{pa 127 759}%
    \special{pa 886 1266}%
    \special{fp}%
    \special{pa 886 1266}%
    \special{pa 633 759}%
    \special{fp}%
    \special{pa 633 759}%
    \special{pa 380 1266}%
    \special{fp}%
    \special{pa 1392 1266}%
    \special{pa 1139 759}%
    \special{fp}%
    \special{pa 1139 759}%
    \special{pa 380 1266}%
    \special{fp}%
    \special{pa 886 1266}%
    \special{pa 1646 759}%
    \special{fp}%
    \special{pn 8}%
    \special{ar 889 697 759 759 0.081754 0.845541}%
    \special{ar 2149 1328 759 759 -3.059838 -2.296052}%
    \graphtemp=.5ex\advance\graphtemp by 1.266in
    \rlap{\kern 2.000in\lower\graphtemp\hbox to 0pt{\hss $Y$\hss}}%
    \graphtemp=.5ex\advance\graphtemp by 0.759in
    \rlap{\kern 2.000in\lower\graphtemp\hbox to 0pt{\hss $X$\hss}}%
    \graphtemp=.5ex\advance\graphtemp by 1.337in
    \rlap{\kern 1.464in\lower\graphtemp\hbox to 0pt{\hss $y$\hss}}%
    \graphtemp=.5ex\advance\graphtemp by 0.658in
    \rlap{\kern 1.646in\lower\graphtemp\hbox to 0pt{\hss $x$\hss}}%
    \special{ar 127 127 127 127 0 6.28319}%
    \special{ar 633 127 127 127 0 6.28319}%
    \special{ar 1139 127 127 127 0 6.28319}%
    \special{pa 127 759}%
    \special{pa 127 253}%
    \special{da 0.051}%
    \special{pa 633 759}%
    \special{pa 633 253}%
    \special{da 0.051}%
    \special{pa 1139 759}%
    \special{pa 1139 253}%
    \special{da 0.051}%
    \graphtemp=.5ex\advance\graphtemp by 0.253in
    \rlap{\kern 0.127in\lower\graphtemp\hbox to 0pt{\hss $\bu$\hss}}%
    \graphtemp=.5ex\advance\graphtemp by 0.253in
    \rlap{\kern 0.633in\lower\graphtemp\hbox to 0pt{\hss $\bu$\hss}}%
    \graphtemp=.5ex\advance\graphtemp by 0.253in
    \rlap{\kern 1.139in\lower\graphtemp\hbox to 0pt{\hss $\bu$\hss}}%
    \hbox{\vrule depth1.367in width0pt height 0pt}%
    \kern 2.000in
  }%
}%
}

\caption{Graph for Example~\ref{badgraph}.}
\end{figure}

%

\end{document}